\documentclass[12pt,reqno]{amsart}
\usepackage{amssymb, amsthm, amsmath, amsfonts,mathrsfs}
\usepackage[margin=1.1in]{geometry}
\usepackage{hyperref}
\hypersetup{colorlinks=true, citecolor=blue, linkcolor=blue, urlcolor=blue, pdfstartview=FitH, pdfauthor=Goldmakher and Pollack, pdftitle=Lagrange Refined}

\begin{document}

\pagestyle{empty}

\parskip0pt
\parindent10pt

\newenvironment{answer}{\color{Blue}}{\color{Black}}
\newenvironment{exercise}{\color{Blue}\begin{exr}}{\end{exr}\color{Black}}

\theoremstyle{plain} 
\newtheorem*{lagrange}{Lagrange's four-square theorem}
\newtheorem*{threesquare}{Legendre--Gauss three-squares theorem}
\newtheorem{theorem}{Theorem}
\newtheorem{prop}[theorem]{Proposition}
\newtheorem{lemma}[theorem]{Lemma}
\newtheorem{cor}[theorem]{Corollary}
\newtheorem{conj}[theorem]{Conjecture}
\newtheorem{funfact}[theorem]{Fun Fact}
\newtheorem*{claim}{Claim}
\newtheorem{question}{Question}
\newtheorem*{conv}{Convention}

\theoremstyle{remark}
\newtheorem{exr}{Exercise}

\theoremstyle{definition}
\newtheorem*{defn}{Definition}
\newtheorem{example}{Example}
\newtheorem{rmk}{Remark}

\newtheorem*{FTA}{Fundamental Theorem of Algebra}
\newtheorem*{FTGT}{Fundamental Theorem of Galois Theory}

\renewcommand{\mod}[1]{{\ifmmode\text{\rm\ (mod~$#1$)}\else\discretionary{}{}{\hbox{ }}\rm(mod~$#1$)\fi}}

\newcommand{\ns}{\mathrel{\unlhd}}
\newcommand{\wt}[1]{\widetilde{#1}}
\newcommand{\wh}[1]{\widehat{#1}}
\newcommand{\cbrt}[1]{\sqrt[3]{#1}}
\newcommand{\floor}[1]{\left\lfloor#1\right\rfloor}
\newcommand{\abs}[1]{\left|#1\right|}
\newcommand{\ds}{\displaystyle}
\newcommand{\nn}{\nonumber}
\newcommand{\im}{\textup{im }}
\renewcommand{\ker}{\textup{ker }}
\renewcommand{\char}{\textup{char }}
\renewcommand{\Im}{\textup{Im }}
\renewcommand{\Re}{\textup{Re }}
\newcommand{\area}{\textup{area }}
\newcommand{\isom}
    {\ds \mathop{\longrightarrow}^{\sim}}
\renewcommand{\ni}{\noindent}
\renewcommand{\bar}{\overline}
\newcommand{\pdiv}{\mid\!\mid}

\newcommand{\Gal}{\textup{Gal}}
\newcommand{\Aut}{\textup{Aut}}
\newcommand{\disc}{\textup{disc}}
\newcommand{\sgn}{\textup{sgn}}

\newcommand{\mattwo}[4]{
\begin{pmatrix} #1 & #2 \\ #3 & #4 \end{pmatrix}
}

\newcommand{\vtwo}[2]{
\begin{pmatrix} #1 \\ #2 \end{pmatrix}
}
\newcommand{\vthree}[3]{
\begin{pmatrix} #1 \\ #2 \\ #3 \end{pmatrix}
}
\newcommand{\vcol}[3]{
\begin{pmatrix} #1 \\ #2 \\ \vdots \\ #3 \end{pmatrix}
}


\newcommand*\wb[3]{%
  {\fontsize{#1}{#2}\usefont{U}{webo}{xl}{n}#3}}

\newcommand\myasterismi{%
  \par\bigskip\noindent\hfill\wb{10}{12}{I}\hfill\null\par\bigskip
}
\newcommand\myasterismii{%
  \par\bigskip\noindent\hfill\wb{15}{18}{UV}\hfill\null\par\medskip
}
\newcommand\myasterismiii{%
  \par\bigskip\noindent\hfill\wb{15}{18}{z}\hfill\null\par\bigskip
}

\newcommand{\one}{{\rm 1\hspace*{-0.4ex} \rule{0.1ex}{1.52ex}\hspace*{0.2ex}}}

\renewcommand{\v}{\vec{v}}
\newcommand{\w}{\vec{w}}
\newcommand{\e}{\vec{e}}
\newcommand{\m}{\vec{m}}
\renewcommand{\u}{\vec{u}}
\newcommand{\vecx}{\vec{e}_1}
\newcommand{\vecy}{\vec{e}_2}
\newcommand{\vo}{\vec{v}_1}
\newcommand{\vt}{\vec{v}_2}

\renewcommand{\o}{\omega}
\renewcommand{\a}{\alpha}
\renewcommand{\b}{\beta}
\newcommand{\g}{\gamma}
\renewcommand{\d}{\delta}
\renewcommand{\t}{\theta}

\newcommand{\Z}{\mathbb Z}
\newcommand{\ZN}{\Z_N}
\newcommand{\Q}{\mathbb Q}
\newcommand{\N}{\mathbb N}
\newcommand{\R}{\mathbb R}
\newcommand{\C}{\mathbb C}
\newcommand{\F}{\mathbb F}
\newcommand{\B}{\mathcal B}
\newcommand{\p}{\mathcal P}
\renewcommand{\P}{\mathbb P}
\renewcommand{\r}{\mathcal R}
\newcommand{\s}{\mathscr S}
\renewcommand{\L}{\mathcal L}
\renewcommand{\l}{\lambda}
\newcommand{\E}{\mathcal E}
\newcommand{\oh}{\mathcal O}

\newcommand{\0}{{\vec 0}}

\newcommand{\ignore}[1]{}

\newcommand{\poly}[1]{\textup{Poly}_{#1}}

\newcommand*\circled[1]{\tikz[baseline=(char.base)]{
            \node[shape=circle,draw,inner sep=2pt] (char) {#1};}}

\newcommand*\squared[1]{\tikz[baseline=(char.base)]{
            \node[shape=rectangle,draw,inner sep=2pt] (char) {#1};}}

\title{Refinements of Lagrange's four-square theorem}
\author{Leo Goldmakher}
\address{Dept of Mathematics \& Statistics,
Williams College,
Williamstown, MA, USA}
\email{leo.goldmakher@williams.edu}
\author{Paul Pollack}
\address{Dept of Mathematics, University of Georgia, Athens, GA, USA}
\email{pollack@uga.edu}

\thanks{LG is partially supported by an NSA Young Investigator grant. PP is partially supported by NSF award DMS-1402268.}

\begin{abstract}
A well-known theorem of Lagrange asserts that every nonnegative integer $n$ can be written in the form $a^2+b^2+c^2+d^2$, where $a,b,c,d \in \Z$. We characterize the values assumed by $a+b+c+d$ as we range over all such representations of $n$.
\end{abstract}

\maketitle

\noindent Our point of departure is the following signature result from a first course in number theory.
\begin{lagrange} Every nonnegative integer can be written as the sum of four integer squares. That is, for every $n \in \N$, there are $a,b,c,d\in \Z$ with \begin{equation}\label{eq:foursquaresrep} n=a^2+b^2+c^2+d^2.\end{equation}
\end{lagrange}
\noindent For instance ($n=2017$), we have
\[ 2017 = 18^2 + 21^2 + 24^2 + 26^2 . \]

Twenty years before Lagrange's proof, Euler had already conjectured a refinement of the four-square theorem for odd numbers $n$. The following statement can be found in a letter to Goldbach dated June 9, 1750.
\begin{conj} Every odd positive integer $n$ has a representation in the form \eqref{eq:foursquaresrep} satisfying the extra constraint $a+b+c+d=1$.
\end{conj}
\noindent Picking back up our earlier example, when $n=2017$, Euler's conjecture is satisfied with $a=-18$, $b=21$, $c=24$, and $d=-26$.
A proof of Euler's conjecture was posted to MathOverflow by Franz Lemmermeyer in September 2010.\footnote{See \url{https://mathoverflow.net/questions/37278/euler-and-the-four-squares-theorem}.}

Much more recently, Sun \& Sun (apparently unaware of Euler's conjecture) presented a number of related refinements of Lagrange's theorem \cite{sunsun16} (cf. \cite{sun16}). One of their many results is that for every $n \in \N$, there is a representation \eqref{eq:foursquaresrep} with $a+b+c+d$ a  square, as well as one with $a+b+c+d$ a cube \cite[Theorem 1.1(a)]{sunsun16}.

We can unify all the above assertions by introducing the \emph{sum spectrum}
\[ \s(n) = \{a+b+c+d: a^2+b^2+c^2+d^2=n\} . \]
Lagrange's theorem is equivalent to $\s(n) \neq \emptyset$; Euler's conjecture asserts that ${1 \in \s(n)}$ for all odd $n \in \N$; and Sun \& Sun's theorem asserts that $\s(n)$ contains a perfect square and a perfect cube for every $n$. Our goal in this note is to completely describe the set $\s(n)$.

We have not found our results stated anywhere in the literature, but we do not claim they are novel. In the introduction to his resolution \cite{cauchy13} of Fermat's polygonal number conjecture, Cauchy poses the following problem:
\emph{D\'{e}composer un nombre entier donn\'{e} en quatre quarr\'{e}s dont les racines fassent une somme donn\'{e}e.}\footnote{\emph{Decompose a given whole number into four squares whose roots make a given sum.}} For Cauchy, ``racines'' are nonnegative; hence, he is asking for a description of
\[ \s^{+}(n):= \{a+b+c+d: a, b, c, d \in \N \text{ and } a^2+b^2+c^2+d^2=n\}. \]
Cauchy goes on to prove a partial characterization of $\s^{+}(n)$ (see Remark \ref{rmk:Cauchy} below for a summary of his results), by essentially the same methods we describe below.
Despite being anticipated,
we believe an explicit description of $\s(n)$ is sufficiently interesting (and Cauchy's work on $\s^{+}(n)$ sufficiently underappreciated) to warrant popularization here. Moreover, we will show how our characterization of $\s(n)$ immediately implies both Sun \& Sun's theorems and a generalization of Euler's conjecture.

We begin by recording two easy observations. First, since an integer and its square have the same parity, every $T \in \s(n)$ satisfies
\begin{equation}\label{eq:easycondition1} T\equiv n\mod{2}. \end{equation}
Second, for any real numbers $a,b,c,d$, the Cauchy--Schwarz inequality yields
\begin{align*} (a+b+c+d)^2 &\le (a^2 + b^2+c^2+d^2) (1^2+1^2+1^2+1^2) \\
&= 4(a^2+b^2+c^2+d^2); \end{align*}
it follows that every $T \in \s(n)$ satisfies
\begin{equation}\label{eq:easycondition2} T^2 \le 4n. \end{equation}

As shown in Table \ref{tbl:exceptional}, the necessary conditions \eqref{eq:easycondition1} and \eqref{eq:easycondition2} are quite often (but not always) sufficient for membership in $\s(n)$. The following theorem, which is our main result, tells the full story.

\begin{table}
\begin{tabular}{|c||c|}
  \hline
  $n$ & exceptional $T$ \\
  \hline\hline
  1 & $\emptyset$ \\
  \hline
  2 & $\emptyset$ \\
  \hline
  3 & $\emptyset$ \\
  \hline
  4 &$\emptyset$ \\
  \hline
  5 & $\emptyset$ \\
  \hline
  6 & $\emptyset$ \\
  \hline
  7 & $\emptyset$ \\
  \hline
  8 & $\{\pm 2\}$ \\
  \hline
  9 & $\emptyset$ \\
  \hline
  10 & $\emptyset$ \\
  \hline
  11 & $\emptyset$ \\
  \hline
  12 & $\emptyset$ \\
  \hline
  13 & $\emptyset$ \\
  \hline
  14 & $\emptyset$ \\
  \hline
  15 & $\emptyset$ \\
  \hline
  16 & $\{\pm 2, \pm 6\}$ \\
  \hline
\end{tabular}
\hspace*{\fill}
\begin{tabular}{|c||c|}
  \hline
  $n$ & exceptional $T$ \\
  \hline\hline
  17 & $\emptyset$ \\
  \hline
  18 & $\emptyset$ \\
  \hline
  19 & $\emptyset$ \\
  \hline
  20 & $\emptyset$ \\
  \hline
  21 & $\emptyset$ \\
  \hline
  22 & $\emptyset$ \\
  \hline
  23 & $\emptyset$ \\
  \hline
  24 & $\{\pm 2, \pm 6\}$ \\
  \hline
  25 & $\emptyset$\\
  \hline
  26 & $\emptyset$\\
  \hline
  27 & $\emptyset$ \\
  \hline
  28 & $\{0\}$ \\
  \hline
  29 & $\emptyset$ \\
  \hline
  30 & $\emptyset$ \\
  \hline
  31 & $\emptyset$ \\
  \hline
  32 & $\{\pm 2, \pm 4, \pm 6, \pm 10\}$ \\
  \hline
\end{tabular}
\hspace*{\fill}
\begin{tabular}{|c||c|}
  \hline
  $n$ & exceptional $T$ \\
  \hline\hline
  33 & $\emptyset$ \\
  \hline
  34 & $\emptyset$ \\
  \hline
  35 & $\emptyset$ \\
  \hline
  36 & $\emptyset$ \\
  \hline
  37 & $\emptyset$ \\
  \hline
  38 & $\emptyset$ \\
  \hline
  39 & $\emptyset$ \\
  \hline
  40 & $\{\pm 2, \pm 6, \pm 10\}$ \\
  \hline
  41 & $\emptyset$ \\
  \hline
  42 & $\emptyset$ \\
  \hline
  43 & $\emptyset$ \\
  \hline
  44 & $\{\pm 8\}$ \\
  \hline
  45 & $\emptyset$ \\
  \hline
  46 & $\emptyset$ \\
  \hline
  47 & $\emptyset$ \\
  \hline
  48 & $\{\pm 2, \pm 6, \pm 10\}$\\
  \hline
\end{tabular}
\vskip 0.05in
\caption{Values of $T$ satisfying \eqref{eq:easycondition1} and \eqref{eq:easycondition2} but not belonging to $\s(n)$.}\label{tbl:exceptional}
\end{table}

\begin{theorem}\label{thm:keyprop} Suppose $n$ and $T$ are integers satisfying \eqref{eq:easycondition1}. Then $T \in \s(n)$
if and only if $4n-T^2$ is a sum of three integer squares.
\end{theorem}\tabularnewline

\noindent Note that \eqref{eq:easycondition2} is implied by the condition on $4n-T^2$ and so does not need to be included explicitly as a hypothesis in Theorem \ref{thm:keyprop}.

To convince the reader that Theorem \ref{thm:keyprop} qualifies as a complete description of $\s(n)$, we recall the following classical result (see the Appendix to Chapter IV of \cite{serre73} for a proof).

\begin{threesquare}
Let $n \in \N$. Then $n$ can be written as a sum of three squares if and only if $n \neq 4^k(8\ell+7)$ for any $k,\ell \in \N$.
\end{threesquare}

\begin{proof}[Proof of Theorem \ref{thm:keyprop}]
We begin by recording the easily-verified identity
\begin{multline}
(2(a+b)-T)^2 + (2(a+c)-T)^2 + (2(b+c)-T)^2 + T^2
\\
= 4a^2 + 4b^2 + 4c^2 + 4(T-a-b-c)^2 .\label{eq:1stID}
\end{multline}
Thus if $T \in \s(n)$, say with $a^2+b^2+c^2+d^2=n$ and $a+b+c+d=T$, then
\[ (2(a+b)-T)^2 + (2(a+c)-T)^2 + (2(b+c)-T)^2
= 4n-T^2, \]
so that $4n-T^2$ is a sum of three squares.

Conversely, suppose that $4n-T^2$ is a sum of three squares, say
\begin{equation}\label{eq:3squaresexpression} 4n-T^2 = A^2 + B^2 + C^2. \end{equation}
In view of \eqref{eq:1stID}, it is enough to show that---after possibly swapping the signs of $A$, $B$, $C$---there are $a, b, c \in \Z$ with
    \begin{equation}\label{eq:abcsystem} 2a+2b-T = A,\quad 2a+2c-T=B,\quad 2b+2c-T = C. \end{equation}
Indeed, in that case setting $d = T-(a+b+c)$, we have $$a+b+c+d = T,$$ and
\begin{align*} 4(a^2+b^2+c^2+&d^2) - T^2 \\&= (2a+2b-T)^2 + (2a+2c-T)^2 + (2b+2c-T)^2 \\&=A^2+B^2+C^2 = 4n-T^2,\end{align*}
so that \[ a^2+b^2+c^2+d^2=n.\] Thus, we focus our attention on \eqref{eq:abcsystem}.

Solving for $a, b, c$ in terms of $A,B,C$ gives
\[ a=\frac{1}{4}(A+B-C+T), \enskip b = \frac{1}{4}(A-B+C+T), \enskip c= \frac{1}{4}(-A+B+C+T). \]
We claim that $A,B,C$, and $T$ must all have the same parity. To see this, note that \eqref{eq:3squaresexpression} gives ${A^2 + B^2 + C^2 \equiv -T^2\mod{4}}$. If $T$ is odd then ${A^2+B^2+C^2 \equiv 3 \mod 4}$, and a moment's thought shows that all of $A, B$, and $C$ must be odd. Similarly, if $T$ is even, then $A^2+B^2+C^2 \equiv 0 \mod 4$, and this forces $A, B$, and $C$ to all be even.
In either case, the difference between any pair of $A$, $B$, and $C$ is even, so the difference between any pair of $a$, $b$, and $c$ is an integer. It follows that if any of $a, b, c \in \Z$, then all three are in $\Z$. Moreover,
\[ A+B-C+T \equiv A+B+C+T \equiv A^2 + B^2 + C^2 + T^2 \equiv 4n\equiv 0 \mod 2, \]
and so the only way we can fail to have $a \in \Z$ (and hence all of $a, b, c \in \Z$) is if
\begin{equation}\label{eq:cong} A+B-C+T \equiv 2\mod{4}. \end{equation}
If $T$ is odd, then $A$ is odd, and so if necessary we can replace $A$ with $-A$ to avoid \eqref{eq:cong}. If $T$ is even, we will show that \eqref{eq:cong} cannot occur. Indeed, \eqref{eq:cong} implies that ${8 \nmid (A+B-C+T)^2}$. But
\begin{align*}
(A+B-&C+T)^2\\&= A^2 + B^2 + C^2 + T^2 +2(AB - AC + AT -BC +BT - CT) \\
            &= 4n + 2(AB-AC+AT-BC+BT-CT) \\
            &\equiv 0 \mod 8;
\end{align*}
here we used that $n \equiv T \equiv 0 \mod{2}$ and that all of $A, B, C$, and $T$ are even.
\end{proof}

Let us see how Theorem \ref{thm:keyprop} makes quick work of both the conjecture of Euler and the theorems of Sun \& Sun. We begin with the latter. If $4n$ itself is a sum of three squares, then Theorem \ref{thm:keyprop} shows that $T=0 \in \s(n)$, and $0$ is both a square and a cube. Otherwise, by the Legendre--Gauss theorem, $4n = 4^{k+1} (8\ell+7)$, where $k$ and $\ell$ are nonnegative integers. Then
\begin{multline*} 4n - (2^{k})^2 = 4^k(32\ell+27), \quad 4n-(2^{k+1})^2 = 4^{k+1}(8\ell+6), \\ \text{and} \quad 4n - (2^{k+2})^2 = 4^{k+1}(8\ell+3); \end{multline*}
invoking the Legendre--Gauss theorem once more, we see that all three of these numbers are sums of three squares. By Theorem \ref{thm:keyprop}, all  of $2^k, 2^{k+1}, 2^{k+2}$ must belong to $\s(n)$. Clearly, the set $\{2^k, 2^{k+1}, 2^{k+2}\}$ contains both a square and a cube.

As for Euler's conjecture, we prove the following generalization (which, for most $n$, gives a very simple description of $\s(n)$):

\begin{prop}\label{prop:no4} Suppose $n \in \N$ is not a multiple of 4.
Then
\[
\s(n) = \{T \equiv n \mod 2 : |T| \leq 2 \sqrt n\} .
\]
\end{prop}

\begin{rmk}\label{rmk:Cauchy}
Cauchy proves that if $T \in \s^{+}(n)$, then $4n-T^2$ is a sum of three squares, and that when $4 \nmid n$,
\[
\s^{+}(n) \supseteq
\{T \equiv n \mod 2 : \sqrt{3n-2} - 1 \le T \le 2 \sqrt n\} .
\]
See \cite[Corollary I of Theorem I, Theorem IV, and Corollary II of Theorem III]{cauchy13}.
\end{rmk}

\begin{proof}[Proof of Proposition \ref{prop:no4}]
In view of Theorem \ref{thm:keyprop}, our task is to show that $4n-T^2$ is a sum of three squares whenever $4 \nmid n$. Suppose first that $n$ is odd, so that $T$ is also odd. Then $4n \equiv 4\mod{8}$ and $T^2\equiv 1\mod{8}$, whence $4n-T^2 \equiv 3 \mod 8$. By the Legendre--Gauss theorem, $4n-T^2$ is a sum of three squares, and we're done.
Now suppose instead that $n$ is twice an odd integer. Then $T$ is even, say $T=2t$, so that $4n - T^2 = 4(n-t^2)$.
It will suffice to show that $n-t^2$ is a sum of three squares, for then $4(n-t^2)$ is as well.
Since $n\equiv 2\mod{4}$ and $t^2\equiv 0\text{ or }1\mod{4}$, we have $n-t^2 \equiv 1\text{ or } 2\mod{4}$. In particular, $n-t^2$ is not of the form $4^k(8\ell+7)$, and so the desired conclusion follows from the Legendre--Gauss theorem. This completes the proof.
\end{proof}

We conclude this note with a few remarks about the structure of $\s(n)$ for general $n$. When $8\mid n$, it is easy to see that any integer solution to
\[ a^2+b^2+c^2+d^2= n \]
has all of $a,b,c,d$ even. Thus, there is a bijection $(a,b,c,d) \leftrightarrow (a/2,b/2,c/2,d/2)$ between representations of $n$ as a sum of four squares and representations of $n/4$. Consequently,
\[ \s(n) = 2 \s(n/4), \]
where the notation on the right-hand side means dilation by a factor of $2$. Iterating, if $k$ is the largest nonnegative integer for which $2^{2k+3} \mid n$, we find that
\[ \s(n) = 2^{k+1} \s(n/4^{k+1}). \]
We have from our choice of $k$ that $2\mid n/4^{k+1}$ while $8\nmid n/4^{k+1}$.

The observations of the last paragraph show that to describe $\s(n)$, it is enough to consider those cases where $8\nmid n$. When $4\nmid n$, Proposition \ref{prop:no4} tells us the answer. However, when $4 \mid n$, it does not seem that there is much to be said beyond what follows immediately from Theorem \ref{thm:keyprop} and the Legendre--Gauss theorem.

The situation becomes both clearer and a bit cleaner if one is willing to shift perspective.
Rather than first picking $n$ and asking for a description of the elements of $\s(n)$, we may pick $T$ and ask for which $n$ we have $T \in \s(n)$.

\begin{prop} Let $T \in \Z$. Assume that $n\ge T^2/4$ and $n\equiv T\mod{2}$.
\begin{enumerate}
  \item If $T$ is odd, then $T \in \s(n)$.
  \item If $T$ is twice an odd integer, then $T \in \s(n)$ if and only if $n\not\equiv 0\mod{8}$.
  \item Suppose that $4\mid T$. Then $T \in \s(n)$ if and only if
  \[ n \notin \bigcup_{k \ge 1} \{T^2-4^k \mod{2^{2k+3}}\}. \]
  Here the right-hand side is an infinite union of disjoint residue classes modulo $2^{2k+3}$, over positive integers $k$.
\end{enumerate}
\end{prop}
\ni
We leave the (routine) proof of this proposition to the interested reader. 


\providecommand{\bysame}{\leavevmode\hbox to3em{\hrulefill}\thinspace}
\providecommand{\MR}{\relax\ifhmode\unskip\space\fi MR }
\providecommand{\MRhref}[2]{%
  \href{http://www.ams.org/mathscinet-getitem?mr=#1}{#2}
}
\providecommand{\href}[2]{#2}

\end{document}